\documentclass[11pt]{article}
\usepackage{latexsym, comment}
\usepackage{url}
\usepackage{epsfig,enumerate}
\usepackage{amsmath,amsthm,amssymb, bbm}
\usepackage[usenames]{color}\usepackage{graphicx}
\numberwithin{equation}{section}
\definecolor{brown}{cmyk}{0, 0.72, 1, 0.45}
\definecolor{grey}{gray}{0.5}

\renewcommand{\epsilon}{\varepsilon}

\newcounter{rot}

\def\a{\alpha}  \def\d{\delta} \def\D{\Delta}
    
  \def\k{\kappa}

\def\t{\tau}

\newtheorem{maintheorem}{Theorem}

\newtheorem*{conjecture*}{Conjecture}
\newtheorem{theorem}{Theorem}[section]
\newtheorem*{theorem*}{Theorem}
\newtheorem{lemma}[theorem]{Lemma}

\newtheorem{fact}[theorem]{Fact}

\newtheorem{definition}[theorem]{Definition}

\newcommand{\bfrac}[2]{\left(\frac{#1}{#2}\right)}

\newcommand{\set}[1]{\left\{#1\right\}}

\def\E{\mathbb{E}}

\def\P{\mathbb{P}}
\def\Pr{\mathbb{P}}

\newcommand{\scr}{\mathcal}
\newcommand{\eps}{\varepsilon}

\newcommand{\ceil}[1]{\left\lceil #1 \right\rceil}
\newcommand{\of}[1]{\left( #1 \right) }

\newcommand{\abs}[1]{\left| #1 \right|}

\newcommand{\sqbracs}[1]{\left[ #1 \right]}
\newcommand{\sqbs}[1]{\left[ #1 \right]}
\newcommand{\braces}[1]{\left\{ #1 \right\}}
\newcommand{\Mean}[1]{\E\sqbs{#1}}
\newcommand{\tbf}[1]{\textbf{#1}}
\renewcommand{\Pr}[1]{\mathbb{P}\left[ #1 \right]}

\allowdisplaybreaks[1]

\newcommand{\ds}[1]{\displaystyle{#1}}
\newcommand{\ignore}[1]{}

\newcommand{\beq}[1]{\begin{equation}\label{#1}}
\newcommand{\eeq}{\end{equation}}

\newcommand{\Bin}{\operatorname{Bin}}

\title{The $t$-tone chromatic number of random graphs}

\author{Deepak Bal\thanks{Department of Mathematical Sciences,
Carnegie Mellon University, Pittsburgh PA15213.} \and Patrick Bennett \thanks{Department of Mathematical Sciences, Carnegie Mellon University, Pittsburgh PA15213.}\and Andrzej Dudek\thanks{Department of Mathematics, Western Michigan University, Kalamazoo, MI}
 \and Alan Frieze\thanks{Research supported in part by NSF Grant CCF2013110}}

\author{Deepak Bal%
\footnote{\footnotesize {Department of Mathematical Sciences, Carnegie Mellon University, Pittsburgh, PA 15213}}
\newcounter{fnnumber}
\setcounter{fnnumber}{\value{footnote}}
\and
Patrick Bennett%
\footnotemark[\value{fnnumber}]
\and
Andrzej Dudek%
\footnote{\footnotesize {Department of Mathematics, Western Michigan University, Kalamazoo, MI 49024}}
\thanks{\footnotesize{Research supported in part by Simons Foundation Grant \#244712}}
\and
Alan Frieze%
\footnotemark[\value{fnnumber}]
\thanks{\footnotesize{Research supported in part by NSF Grant CCF2013110}}
}

\begin{document}
\maketitle

\begin{abstract}
 A proper 2-tone $k$-coloring of a graph is a labeling of the vertices with elements from $\binom{[k]}{2}$ such that adjacent vertices receive disjoint labels and vertices distance 2 apart receive distinct labels. The 2-tone chromatic number  of a graph $G$, denoted $\t_2(G)$ is the smallest $k$ such that $G$ admits a proper 2-tone $k$ coloring. In this paper, we prove that w.h.p. for $p\ge Cn^{-1/4}\ln^{9/4}n$,  $\t_2(G_{n,p})=(2+o(1))\chi(G_{n,p})$ where $\chi$ represents the ordinary chromatic number. For sparse random graphs with $p=c/n$, $c$ constant, we prove that $\t_2(G_{n,p}) = \ceil{\of{\sqrt{8\D+1} +5}/{2}}$ where $\D$ represents the maximum degree. For the more general concept of $t$-tone coloring, we achieve similar results.
\end{abstract}

\section{Introduction}
The ordinary chromatic number of a graph $G$, denoted $\chi(G)$ is the fewest number of colors necessary to label the vertices of $G$ such that no two adjacent vertices receive the same color.  There have been many generalizations of this concept, for example list coloring, $t$-set coloring \cite{BT79}, and distance-$t$ colorings \cite{CGH}.  A natural extension which further generalizes the concepts mentioned above is that of a $t$-tone coloring. Chartrand introduced $t$-tone coloring as a generalization of proper coloring, which is equivalent to 1-tone coloring. The concept was initially studied in a research group directed by Zhang \cite{FGPS} and then investigated by Bickle and Phillips \cite{BP}. 

Throughout the paper, if $\ell\le k$  are positive integers, $[k]$ refers to the set $\set{1,2,\ldots,k}$ and $\binom{[k]}{\ell}$ refers to the collection of $\ell$ sized subsets of $[k]$. For vertices $u$ and $v$ of $G$, $d(u,v)$ refers to the distance between $u$ adnd $v$, \emph{i.e.} the minimum number of edges on a path between $u$ and $v$. We may now give the formal definition of a $t$-tone coloring which appears in \cite{FGPS} and \cite{BP}.
\begin{definition}
 Let $G=(V,E)$ be a graph and let $t$ be a positive integer. A (proper) $t$-tone $k$-coloring of a graph is a function $f:V(G)\rightarrow \binom{[k]}{t}$ such that $\abs{f(u) \cap f(v)} < d(u,v)$ for all distinct vertices $u$ and $v$. A graph that admits a $t$-tone $k$-coloring is $t$-tone $k$-colorable. The $t$-tone chromatic number of $G$, denoted $\t_t(G)$ is the least integer $k$ such that $G$ is $t$-tone $k$-colorable. 
\end{definition}

For a vertex $v\in V(G)$, we call $f(v)$ the \emph{label} on $v$. The elements of $f(v)$ are colors. In this paper, we are concerned primarily with the 2-tone chromatic number, $\t_2(G).$ Note that the definition in this case says that adjacent vertices receive disjoint labels and vertices at distance 2 receive distinct labels. The classical Erd\H os-R\'enyi-Gilbert random graph $G_{n,p}$ is a graph on vertex set $[n]$ in which each potential edge in $\binom{[n]}{2}$ appears independently with probability $p$. We say that an event occurs \emph{with high probability}, denoted w.h.p, if the probability of the event tends to 1 as $n$ tends to infinity. The main results of this paper concern $\t_2(G_{n,p})$ in 2 ranges of $p$.
On the dense end of the spectrum, we have the following result.
\begin{maintheorem}\label{densemainthm}
 Let $p=p(n)$ satisfy $Cn^{-1/4}\ln^{9/4}n \le p < \eps < 1$ where $C$ is a sufficiently large constant and $\eps$ is any constant $<1$. Then w.h.p., 
\[\t_2(G_{n,p}) = (2+o(1))\chi(G_{n,p}).\]
\end{maintheorem}
For sparse random graphs we prove the following:
\begin{maintheorem}\label{sparsemainthm}
Let $c$ be a constant, and let $p=c/n$. If we let $\D$ represent maximum degree, then w.h.p.,\[\t_2(G_{n,p}) = \ceil{\frac{\sqrt{8\D + 1} + 5}{2}}.\]
\end{maintheorem}

In the dense range w.h.p., the diameter of $G_{n,p}$ is 2. Thus finding a $t$-tone coloring of the random graph in this range amounts to finding labels which are disjoint on adjacent vertices and intersect in at most one color on non-adjacent pairs. For this reason, our proof techniques for the $t=2$ case may be easily extended to the $t\ge 3$ case. Our result in the sparse case relies on another tight result
for 2-tone colorings of trees. There is no known analogous tight result for $t$-tone colorings with $t\ge 3$. Hence the result we have in the sparse case for $t\ge3$ is weaker. The results for $t\ge 3$ appear as Theorems \ref{densemainthmgeneral} and \ref{sparsemaingen} in Section \ref{general}.

\section{A lower bound on $\t_t(G)$}

Consider a $t$-tone $k$-coloring of any graph $G$ on vertex set $[n]$, and for each $i \in [k]$, let $S_i$ be the set of vertices that have color $i$ as one of their $t$ colors. When we sum $|S_i|$ over $i$, each vertex is counted $t$ times (once for each color it has). Thus 
$$tn = \sum_i |S_i| \le k \cdot \a \of{G}$$
so $\t_t(G)  \ge \frac{tn}{\a \of{G}}$.
The above inequality, together with known bounds on $\chi\of{G_{n,p}}\sim n/\a\of{G_{n,p}}$  give us a lower bound on the $t$-tone chromatic number of $G_{n,p}$. In particular, 
\[
 \t_t(G_{n,p})  \ge (t-o(1))  \chi(G_{n,p}) 
\]
w.h.p..

\section{Upper bound for dense case}

Throughout this section, let $G=G_{n,p}$ on vertex set $[n]$ and let $Cn^{-1/4}\ln^{9/4}n \le p < 1$ where $C$ is a sufficiently large constant. We adapt the proof strategy of Bollob\'as \cite{BolBook} for obtaining bounds on the ordinary chromatic number $\chi(G)$. Bollob\'as' strategy requires two key facts. First, one shows that w.h.p. every sufficiently large subgraph has an independent set almost as large as $\a \of{G_{n,p}}$. Then we show that w.h.p. there are no small subgraphs with high edge density. The strategy for coloring $G_{n,p}$ is as follows: iteratively find a maximum independent set in the graph and remove the vertices, until the remaining set of vertices is sufficiently small. The remaining graph does not have high edge density. Thus we may greedily color the rest of the vertices using new colors (and not very many of them). W.h.p. the resulting coloring uses $\frac{n}{\a \of{G_{n,p}}} (1+o(1))$ colors, which is clearly asymptotically optimal.

If we want a $2$-tone coloring, we may begin by giving $G_{n,p}$ an ordinary proper coloring as above. But then we have to assign each vertex another color, and the colors we assign in this second pass must be carefully chosen with regard to the colors that are already there. 
From the lower bound, we know that we will need to use at least (roughly) twice the number of colors we would need for an ordinary coloring. So for our second pass we might as well use new colors (\emph{i.e.} none of the same colors we used in the first pass). Also, w.h.p. the diameter of $G_{n,p}$ is $2$ for this range of $p$, so in our final $2$-tone coloring we cannot assign any two vertices the same pair of colors.

Let $\mathcal{P}=\left\{P_1, P_2, \ldots , P_a \right\}, \mathcal{R}=\left\{R_1, R_2, \ldots , R_b \right\}$ be partitions of $[n]$.
We will say that a set $\mathcal{R}$ \emph{respects} $\mathcal{P}$ if $|R_i \cap P_j| \leq 1$ for all $i,j$. We also say that a specific set $R$ respects $\scr{P}$ if $\abs{R\cap P_j} \le 1$ for all $j$.
We can find a $2$-tone coloring of $G_{n,p}$ by finding two ordinary colorings such that the partitions generated by the color classes respect each other. We will accomplish that task with the proof strategy discussed above in mind. Start with a partition $\mathcal{P}$ of $[n]$ into sets of vertices that are independent in $G_{n,p}$ (i.e. the parts of $\mathcal{P}$ are color classes of a proper coloring). Then iteratively find large independent sets that respect  $\mathcal{P}$ and remove those vertices from the graph. Once the remaining graph is sufficiently small, it has low enough edge density to be colored greedily using all new colors without having a significant impact on the total number of colors used. 

For ease of notation, set $b:=\frac{1}{1-p}$, and set $k:= \lceil 3 \log_b n  \rceil = \ceil{ \frac{3 \ln n }{\ln b}}$. The following bounds are well known (see e.g. \cite{BolBook},\cite{JLR}): w.h.p. $\a \of{G_{n,p}} < k$ and $\chi \of{G_{n,p}} < \frac{3n}{k}$. 
The two key lemmas we will require to prove Theorem \ref{densemainthm} are as follows:

\begin{lemma}\label{denselem1}
 W.h.p. for every partition $\mathcal{P}$ of $[n]$ into parts of size at most $k$, and every set  $U \subset [n]$ of size $|U| > \frac{n}{\ln^2 n}$, $G[U]$ contains an independent set of size at least $s_0: = \lceil 2 \log_b n - 2 \log_b \log_b n - 5\log_b \ln n \rceil = \ceil{ \frac{2 \ln n + 2 \ln \ln b  - 7 \ln \ln n}{\ln b}}$ which respects $\mathcal{P}$.
\end{lemma}

\begin{lemma}\label{denselem2}
W.h.p.  for every set $H \subset [n]$ of size at most $\frac{n}{\ln^2 n}$, $G[H]$ has at most $\frac{n|H|}{k \ln n}$ edges.
\end{lemma}

Assuming the truth of the lemmas, the proof of Theorem \ref{densemainthm} is as follows:

\begin{proof}[Proof of Theorem \ref{densemainthm}] 
 We start with an ordinary coloring of the vertices using $\of{1+o(1)}\chi \of{G_{n,p}}$ colors, where the partition $\mathcal{P}$ given by the color classes has at most $\frac{3n}{k}$ parts, and each part is of size at most $k$. W.h.p. such a coloring exists.

Now we apply Lemma \ref{denselem1} iteratively, finding large independent sets that respect $\mathcal{P}$ and removing the independent sets from the graph, until less than $\frac{n}{\ln^2 n}$ vertices remain. Let $V'$ be the set of vertices remaining at this point.

We will use a new set of colors for the vertices of $V'$. All we have to do is make sure that the color classes within $V'$ respect the partition $\mathcal{P}$. Thus, the problem of coloring $V'$ is equivalent to finding an (ordinary) coloring of the graph $G'$ with vertex set $V'$, and with the edge set being the union of $E\of{G_{n,p}[V']}$ and the set of edges with both endpoints in the same part of $\mathcal{P}$. The latter set of edges guarantees that no two vertices in the same part of $\mathcal{P}$ will be assigned the same color. Thus, any proper ordinary coloring of $G'$ will serve as a valid completion of our $2$-tone coloring of $G_{n,p}$.

Now the chromatic number of $G'$ is at most its coloring number (see, \emph{e.g.}, Prop. 5.2.2 in \cite{Diestel}), which is at most \[1+ \max \set{\frac{2\abs{E(G'[H])}}{|H|}: H \subset V' }\] where $E(G)$ represents the edge set of a graph $G$. But by Lemma \ref{denselem2}, w.h.p. for all $H \subset V'$ we have $\frac{2\abs{E(G_{n,p}[H])}}{|H|} \le \frac{2n}{k \ln n}$. Of course, $G'[H]$ has some edges that are not in $E(G_{n,p}[H])$. Specifically, $G'[H]$ has all possible edges with both endpoints in the same part of $\mathcal{P}$. Using Jensen's inequality, the convexity of the function $\binom{x}{2}$, and the properties of $\mathcal{P}$, we see that $G'[H]$ has at most $\frac{|H|}{k} \binom{k}{2} = O(\abs{H}k)$ such edges. Therefore for any $H\subset V'$, we have $\frac{2\abs{E(G'[H])}}{|H|} \le \frac{2n}{k \ln n} + O(k)= o(\chi(G_{n,p}))$. Thus we can color $G'$ using a negligible number of colors.
\end{proof}

We now prove the two lemmas.

\begin{proof}[Proof of Lemma \ref{denselem1}]

Fix $\mathcal{P}=\left\{P_1, P_2, \ldots , P_m \right\}$, a partition of $[n]$ into $m< \frac{3n}{k}$ parts of size at most $k$, and let the random variable $X$ be the number of independent sets of size $s:=\lceil 2 \log_b n - 2 \log_b \log_b n - \log_b \ln n \rceil = \ceil{\frac{2 \ln n + 2 \ln \ln b  - 3 \ln \ln n}{\ln b}} $ respecting $\mathcal{P}$. Note that $s>s_0$ from the statement of Lemma \ref{denselem1}. By an application of Janson's inequality (see, \emph{e.g.}, inequality 2.18(ii) in \cite{JLR}), 
\[
 \Pr{X=0} \le \exp \of{- \frac{ \E{\sqbracs{X}}^2}{\sum_{S, S'} \E \sqbracs{{X_S X_{S'}}}}}
\]
where the sum in the denominator is taken over all pairs of sets of vertices $S, S'$ of size $s$ such that $|S \cap S'| \ge 2$ and $S,S'$ respect $\scr{P}$. The random variable $X_S$ is just a $0-1$ indicator for whether $S$ is independent in $G_{n, p}$.

Note that the number of $s$-sets not respecting $\mathcal{P}$ is at most 

\begin{equation}
 \sum_{1 \le i \le m} \binom{P_i}{2} \binom{n}{s-2} = O\of{ m s^2 \binom{n}{s-2}} =  O\of{\frac{s^3}{n}\binom{n}{s}}
\end{equation}
so 
\[
 \E{X} \ge \binom{n }{s} (1-p)^{\binom{s }{ 2}} \of{1-O\of{\frac{s^3}{n}}}.
\]

Now we would like to put an upper bound on the sum in the denominator, $\sum_{S, S'} \E{\sqbracs{X_S X_{S'}}}$. We begin by ignoring the fact that the sum is only taken over pairs $S, S'$ respecting $\mathcal{P}$. Thus
\begin{align*}
 \sum_{S, S'} \E{\sqbracs{X_S X_{S'}}} &\le \sum_{2 \le i \le s} \binom{n }{ s} \binom{s }{ i} \binom{n-s }{ s-i}(1-p)^{2\binom{s }{ 2} - \binom{i }{ 2}} \\
&= \binom{n }{ s}^2 (1-p)^{2\binom{s }{ 2}} \sum_{2 \le i \le s} a_i
\end{align*}
where
\[a_i := \frac{ \binom{s }{ i} \binom{n-s }{ s-i}(1-p)^{ - \binom{i }{ 2}}}{\binom{n }{ s}}.\]
Here we note the bounds
\begin{equation}
 a_2 = \Theta\of{\frac{s^4}{n^2}} \qquad\textrm{ and }\qquad a_3 = \Theta\of{\frac{s^6}{n^3}}.
\end{equation}
To estimate the sum $\sum_{2 \le i \le s} a_i$, we define \[r_i := \frac{a_{i+1}}{a_i} = \frac{(s-i)^2}{(i+1)(n-2s+i+1)}(1-p)^{-i}.\]
Analyzing $r_i$ will help us to analyze $a_i$. For example,  since $r_2 = O\of{\frac{s^2}{n}} <1$, we have that $a_2 > a_3$.

Now for $2\le x<s$, define 
\[f(x) := \ln r_x = 2 \ln(s-x) - \ln(x+1) - \ln(n-2s+x+1) + x \ln \of{\frac{1}{1-p}}\]
and note that 
\[f'(x) = -\frac{2}{s-x} - \frac{1}{x+1} - \frac{1}{n-2s+x+1} + \ln \of{\frac{1}{1-p}}.\]
The first term is negative, but negligible unless $x$ is close to $s$. The second term is negative, but negligible unless $x$ is small. The third term is always negligible as $s=o(n)$. The fourth term is positive and constant with respect to $x$. Therefore, $\set{x\,:\, f'(x) > 0}$ is an interval. So the set $\set{x\,:\, f(x) > 0}$ is also an interval and $\set{i\,:\,r_i > 1} = \set{i\,:\, a_{i+1} > a_i}$ is a set of consecutive integers.

Therefore the largest term $a_i$ is either $a_2$ or $a_{i^*}$ where \[i^* := 1 + \max \{i: r_i > 1\}.\]
 Also, the second largest term is one of $a_2, a_3, a_{i^*}, a_{i^* -1}$ or $a_{i^* + 1}$.
To estimate $i^*$, define $i':= \lceil s\of{1-\frac{1}{\ln n}}\rceil$, and note that 
\begin{align*}
r_{i'-1} &\ge \frac{(s-i'+1)^2}{i' n} \cdot b^{i'}\\
 &\ge \frac{ \of{\frac{s }{\ln n}}^2}{\sqbracs{ s\of{1-\frac{1}{\ln n}}+1} n} \cdot \of{ \frac{n^2} {\log_b ^{2} n \ln n}}^{1-\frac{1}{\ln n}}\\
&\gg 1,
\end{align*}
so $i^* \ge i'$. Now we will estimate $a_i$ for $i \ge i'-1$. First, 
\begin{align*}
a_s = \frac{1}{\binom{ n }{ s} (1-p)^{\binom{s }{ 2}}} &\le \of{\frac{s}{n}}^s b^{s\of{\frac{s-1}{2}}}\\
 &= \exp \left\{s \sqbracs{ \ln s - \ln n + \frac{s-1}{2} \ln b} \right\}\\
&= \exp \bigg\{s \bigg[ (1+o(1))\ln \ln \of{  n \ln b} - \ln \ln b - \ln n \\
&\left.\left.\quad- \frac{1}{2} \ln b + \frac{1}{2} \ln b \of{  \frac{ 2 \ln n+ 2 \ln \ln b - 3 \ln \ln n }{\ln b} }\right] \right\}\\
&= \exp \left\{ -\Omega \of{s \ln \ln n} \right\}= \exp \left\{ -\Omega \of{\ln n \ln \ln n} \right\}\\
 &< a_3
\end{align*}
Now for any $i'-1 \le i < s$ we have

\begin{align*}
\displaystyle{a_i = \frac{\binom{s}{i} \binom{n-s }{ s-i}(1-p)^{ - \binom{i }{ 2}}}{\binom{n }{ s}}} &\le \displaystyle{ \binom{s }{ s-i} \binom{n-s }{ s-i} a_s } \\ 
&\le \displaystyle{ \of{\frac{e^2 n s}{(s-i)^2}}^{s-i} a_s } \\
&\le \ds{ \of{e^2 n s}^{\frac{s}{\ln n}+1} a_s } \\
&\le\exp\set{\of{\frac{s}{\ln n}+1}\of{2+\ln n + \ln s} - \Omega\of{s\ln\ln n}}\\
&\le\ds{ \exp \left\{ -\Omega \of{\ln n \ln \ln n} \right\}}\\
&< a_3
\end{align*}
Therefore, the largest of the $a_i$ is $a_2$, and the second largest is $a_3$. In particular, 
\[
 \sum_{2 \le i \le s} a_i \le a_2 + s a_3 = O\of{\frac{s^4}{n^2}}.
\]
Thus, 
\begin{equation}\label{expineq}
 \Pr{X=0} \le \exp \left\{- \Omega\of{ \frac{n^2}{s^4} } \right\}.
\end{equation}

 Let $B$ be the number of pairs $(\mathcal{P}, U)$ of partitions $\mathcal{P}$ and sets $U$ for which the lemma fails. We will bound $\Mean{B}$ using a union bound, linearity of expectation, and inequality \eqref{expineq}. Note that since $\frac{n}{\ln^2 n} < |U| \leq n$, and since we are looking for independent sets of size 
\[
 s_0=\lceil 2 \log_b n - 2 \log_b \log_b n - 5\log_b \ln n \rceil \le \left\lceil 2 \log_b \of{|U|} - 2 \log_b \log_b \of{|U|} - \log_b \ln \of{|U|} \right\rceil
\]
within $G[U]$, the inequality \eqref{expineq} applies. Now for fixed $\mathcal{P}, U$, the probability that $G[U]$ has no independent set of size $s_0$ respecting $\mathcal{P}$ is at most 
\[
 \exp \left\{- \Omega\of{ \frac{|U|^2}{s_0^4} } \right\} \le  \exp \left\{- \Omega\of{ \frac{n^2}{s^4 \ln^4 n} } \right\}.
\]
Thus,
\[
 \Mean{B } \le \of{\frac{3n}{k}}^n \cdot 2^n \cdot \exp \left\{- \Omega\of{ \frac{n^2}{s^4 \ln^4 n} } \right\}.
\]
which is $o(1)$ as long as $p\ge Cn^{-1/4}\log^{9/4}$ for $C$ a sufficiently large constant.
\end{proof}

Here's the proof of the second lemma:
\begin{proof}[Proof of Lemma \ref{denselem2}]

First note that for $|H| \le n^\frac{1}{2}$, we are done since $G[H]$ can only have $\binom{|H|}{2} = o\of{\frac{n|H|}{k \ln n}}$ edges. So we turn our attention to larger sets $H$.

Recall the Chernoff bound:
$$\P[Bin(n, p) > (1+\d)np] < \exp \of{-\frac{\d np}{2}}$$ for all $\d>2$. This is a slightly modified version of (2.5) from \cite{JLR}. From this we may deduce that
$$\P \left[E(G[H]) > (1+\d)\binom{\abs{H}}{2}p \right] < \exp \of{-\frac{\d \binom{\abs{H} }{ 2}p}{2}}.$$
Setting $(1+\d)\binom{|H|}{2}p = \frac{n|H|}{k \ln n}$ and solving for $\d$ yields
$$\d = \frac{2n}{(\abs{H}-1)kp \ln n} -1 = \Omega \of{\frac{n}{|H| \ln^2 n}}.$$
So for any fixed $H$ with $ n^\frac{1}{2} < |H| <\frac{n}{\ln^2 n}$, the probability that $G[H]$ has too many edges is at most $\exp \of{-\Omega\of{\frac{n |H| p}{ \ln^2 n}}} = o\of{2^{-n}},$ so w.h.p. there are no such sets $H$. 
\end{proof}

\section{Sparse graphs ($p=c/n$)}
The overall plan for $G=G_{n,p}$ with $p=c/n$, $c$ constant, is to first 2-tone color a set of vertices that includes high degree vertices and two neighborhoods. We will show this set is a forest and then apply the result of \cite{FGPS} which says the 2-tone chromatic number of a tree, $T$ with maximum degree $\D$ is 
\begin{equation}\label{treechromatic}\t_2(T) = \ceil{\frac{\sqrt{8\D+1} +5}{2}} =: \k_\D\end{equation} 
The remaining vertices will be easier to color. This process will yield a proof of Theorem \ref{sparsemainthm}.

Let the vertex set of $G$ be $V$, let $b_0=\ln^{1/4}n$ and let \[V_0 := \braces{v\in V \,:\, deg(v) \ge b_0}.\]
For $k\ge 1$, let
\begin{equation}\label{Vkdefn}V_k := V_0 \cup \bigcup_{i=1}^{k}N^{i}\of{V_0}\end{equation}
where $N^i(Z)$ for $i\ge1$ represents the set the vertices whose distance to vertex set $Z$ is $i$.
Let $H$ represent $G\sqbs{V_2}$, the graph induced on the vertex set $V_2$.  

In the following proofs, we will make use of the configuration model (defined below) on a ``typical'' degree sequence. This is defined as follows:
\begin{definition}
A degree sequence $(d_1,d_2,\ldots,d_n)$ is called \emph{typical} if the following three properties hold:
\begin{enumerate}
\item $\frac{1}{2}\sum_{i=1}^{n}d_i \ge \frac{cn}{3}$,
\item $\ln^{3/4}n \le \max_{1\le i\le n}d_i \le \ln n$,
\item $\abs{\set{i\,:\,d_i\ge b_0}} \le n\ln n\exp\set{-\ln^{1/4}n}$.
\end{enumerate}
\end{definition}

Such degree sequences are called typical because 
\begin{lemma}\label{Gtypical}
With probability $1-o(1)$, the degree sequence of $G$ is typical.
\end{lemma}
\begin{proof}
Property 1 follows immediately from the Chernoff inequality.  $\frac{1}{2}\sum_{i=1}^{n}d_i= |E(G)|$ is the sum of Bernoulli random variables, which is concentrated around its mean, $\frac{cn}{2}$. It is well known \cite{BolBook} that the maximum degree of $G_{n,p}, p=c/n$ is $\Theta\of{\frac{\ln n}{\ln\ln n}}$ with probability $1-o(1)$, so Property 2 holds. Note that the set $V_0$ is the same as the set on the left hand side of property 3. We have
\[\E\sqbs{\abs{V_0}}=n\Pr{\Bin\of{n-1,\frac{c}{n}} \ge \ln^{1/4}n} \le n\exp\set{-\ln^{1/4}n}\] by Chernoff's inequality (see, \emph{e.g.}, Corollary 2.4 in \cite{JLR}).  Consequently, Markov's inequality yields
\[\Pr{\abs{V_0} >n\ln n\exp\set{-\ln^{1/4}n} } =o(1).\]
\end{proof}

\begin{lemma}\label{HisForest} W.h.p. $H$ is a forest. 
\end{lemma}
\begin{proof}

We will prove that w.h.p. $H$ does not have large components. Once that is established, the lemma will follow from a short calculation. We will use the following definition.
\begin{definition}
 For a graph $G$ and integer $i \ge 1$, let the graph $G^i$ have vertex set $V(G)$, and edge set 
\[E(G^i)=\left\{ \{u,v\}: d_G (u,v) \le i \right\}\]
\end{definition}

Our motivation for considering $G^i$ is as follows. Suppose $K$ is a connected component of $H$. Then the the set of vertices $V(K) \cap V_0$ induces a connected component in $G^5$. 
We claim that w.h.p. $G$ has the following properties:

\begin{enumerate}
\renewcommand{\theenumi}{{\bf P\arabic{enumi}}}
\item\label{noHighlyConnectedY} There does not exist $S\subseteq V_0$ such that $\abs{S} \ge s = \ln^{7/8}n$ and $S$ induces a connected component in $G^{5}$. 
\item\label{maxcomponent} The maximum component size in $H$ is at most $\ln^{17/8}n$.
\end{enumerate}

To establish \ref{noHighlyConnectedY} and \ref{maxcomponent}, fix a typical degree sequence $\tbf{d}=(d_1,d_2,\ldots,d_n)$. A random (multi-) graph with degree sequence $\tbf{d}$ is constructed using the configuration model as described in Bollob\'{a}s \cite{Bol80}. Let $m=(d_1+\cdots+d_n)/2$. We construct a random pairing $F$ of the points $W=\bigcup_{i=1}^{n}W_i,\,\abs{W_i}=d_i$ and interpret them as edges of a multi-graph on $[n]$. With a typical degree sequence, the probability that the resulting graph is simple is bounded away from 0 by a function of $c$ and not $n$ (see, \emph{e.g.} \cite{MW91}). We will prove that these three properties hold conditional on a specific degree sequence, and then sum over all degree sequences to get the result unconditionally.

To prove \ref{noHighlyConnectedY}, suppose that such an $S\subseteq V_0$ exists. Then we may assume that $\abs{S}=s$ and that there exists a tree $T$ in $G$ such that the leaves of the tree are a subset of $S$ and $\abs{T} \le 5s$. We may make this assumption on $\abs{T}$ since $G^5[S]$ is connected and each edge in $G^5$ corresponds to a path of length at most 5.  Then
\begin{align*}
\Pr{\neg \ref{noHighlyConnectedY}\mid \tbf{d}} &\le \sum_{t=s}^{5s}\abs{V_0}^sn^{t-s}t^{t-2}\binom{t\D}{2t}\prod\limits_{i=1}^{t-1}\frac{1}{2m - 2i+1}\\
&\le\sum\limits_{t=s}^{5s}n^t\ln^sn\exp\set{-s\ln^{1/4}n}t^{t-2}\bfrac{\D e}{2}^{2t}\bfrac{1}{\frac{2cn}{3}-10\ln^{7/8}n}^{t-1}\\
&\le n\cdot 5\ln^{7/8}n \cdot \exp\set{-\ln^{9/8}n + O(\ln^{7/8}n\ln\ln n) }\\
&=o(1).
\end{align*}
To see the first inequality here, note that $\abs{V_0}^sn^{t-s}$ is an upper bound on the number of ways to choose the vertices of $T$. $t^{t-2}$ is the number of trees on these vertices by Cayley's formula. The number of ways to choose configuration points corresponding to a specific tree is bounded above by $\binom{t\D}{2t}$ since there are at most $t\D$ configuration points and $2(t-1)$ half-edges in $T$.  The last product is the probability that those specific configuration points are paired off in the prescribed manner. 

To prove \ref{maxcomponent}, let $C$ be a component of $H$ and let $K=C\cap V_0$. Then $\abs{C} \le \of{\abs{K}\ln(n)}\ln^{1/4}n$ since $\abs{N(K)} \le \abs{K}\ln(n)$ and each of these vertices may have at most $\ln^{1/4}n$ neighbors. But by \ref{noHighlyConnectedY}, $\abs{K} \le \ln^{7/8}n$. So $\Pr{\neg \ref{maxcomponent}\mid \tbf{d}} = o(1)$.

Let \textbf{P} be the property that $H$ is a forest. Using these two facts we may prove \textbf{P} holds with high probability. We perform breadth first search to reveal $H$ in the following manner.  We reveal the pairs of $F$, one a time starting with pairs with at least one endpoint in $\bigcup_{i:\abs{W_i}\ge b_0}W_i$.  After this, the vertices of $N(V_0)$ have been revealed. We then reveal pairs of $F$ involving points corresponding to vertices of $N(V_0)$ which reveals $N^2(V_0)$.  Lastly, reveal pairs where both endpoints correspond to vertices from $N^2(V_0).$ At this point $H$ has been revealed. Each time an edge is revealed, there is some probability that it closes a cycle. This probability is bounded above by
\[\D\ln^{17/8}n\bfrac{1}{2m-o(n)}\]
since there are at most $\D\ln^{17/8}n$ configuration points corresponding to any particular component. Since $\abs{V_0} \le n\ln n\exp\set{-\ln^{1/4}n}$ and $\D \le \ln n$, we have that $\abs{V_2} \le n\ln^3n\exp\set{-\ln^{1/4}n}$. There are at most $\abs{V_2}\D$ exposures total, so the union bound gives
\begin{align*}
 \Pr{\neg \textbf{P}  \mid \tbf{d}} &\le \abs{V_2}\ln^{17/8}n\D^2\bfrac{1}{2m-o(n)}\\
	&\le \bfrac{3}{c}\ln^{57/8}n\exp\set{-\ln^{1/4}n}\\
	&=o(1).
\end{align*}
Now to remove the conditioning on $\tbf{d}$, we sum up over valid degree sequences.
\begin{align*}
 \Pr{\neg \textbf{P} } &\le \Pr{\tbf{d} \textrm{ not typical}} +  \sum_{\tbf{d}\textrm{ typical}}\Pr{\neg \textbf{P}  \mid \tbf{d}}\Pr{\tbf{d}}\\
 &= o(1)
\end{align*}
by Lemma \ref{Gtypical} and the fact that a weighted average of $o(1)$ terms is $o(1).$
\end{proof}

We now prove Theorem \ref{sparsemainthm} by showing how to color the graph.
\begin{proof} [Proof of Theorem \ref{sparsemainthm}]  By Lemma \ref{HisForest}, $H$ is a forest with probability $1-o(1)$. So
 by \eqref{treechromatic}, we may color $H$ with $\k_\D$ many colors where $\D$ is the maximum degree of $G$.  Give $H$ such a 2-tone coloring and then remove the colors on the vertices of $N^2(V_0).$  This leaves a proper 2-tone coloring on the vertices of $V_1$ (recall \eqref{Vkdefn}).  We will now show that the coloring on $V_1$ can be greedily extended to a proper 2-tone coloring of $G$ without using any new colors.  

Note that any pair of vertices in $V_1$ at distance 1 in $G$ receive disjoint pairs of colors. Any pair of vertices in $V_1$ at distance 2 in $G$ receive distinct pairs of colors. This was the reason for properly coloring $H$ and then uncoloring $N^2(Y).$ Not every proper coloring of $G[V_1]$ can be extended to a proper coloring of $V$ since there may be 2 vertices in $V_1$ at distance 2 in $G$ which are not distance 2 in $G[V_1].$

Let $v$ be an uncolored vertex. We must ensure that the label we assign to $v$ is disjoint from any current labels on $v$'s neighbors and is distinct from any current labels on vertices at distance 2 from $v$.  Let us count the number of labels that we are not allowed to put on $v$. Since $v\not\in Y$, $deg(v) < b_0$. So the number of labels forbidden by $N(v)$ is at most $2b_0\k_\D$. To see this note that at most $2b_0$ colors appear on vertices in $N(v)$ and each of these colors gives rise to $\k_\D -1$ labels which cannot be put on $v$.  Since $v\not \in N(Y)$, $\abs{N^2(v)} \le b_0^2$. So the number of labels forbidden by $N^2(v)$ is at most $b_0^2$, one for each label currently on a vertex of $N^2(v)$.  

So we have that the number of forbidden labels on $v$ is at most
\[2b_0\k_\D + b_0^2 < \binom{\k_\D}{2}.\] Hence there exists a pair of colors that we may use to label $v$. 
\end{proof}

\section{Results for $\t_{t}(G_{n,p})$, $t\ge 3$}\label{general}

\subsection{Dense case}

Our main theorem for dense random graphs and general $t$ is a direct generalization of the $t=2$ case.

\begin{maintheorem}\label{densemainthmgeneral}
 Let $p=p(n)$ satisfy $Cn^{-1/4}\ln^{9/4}n \le p < \eps < 1$ where $C$ is a sufficiently large constant and $\eps$ is any constant $<1$. Then w.h.p., 
\[\t_t(G_{n,p}) = (t+o(1))\chi(G_{n,p}).\]
\end{maintheorem}

\begin{proof}[Proof sketch for Theorem 3] We show that w.h.p. we can find $t$ partitions of $[n]$, $\mathcal{P}_1, \ldots, \mathcal{P}_t$, where each partition consists of $(1+o(1)) \chi(G)$ many independent sets in $G$, and each partition respects each other partition. Once we find $\mathcal{P}_1, \ldots ,\mathcal{P}_t$, we assign $t$ colors to each vertex $v$, according to which part of each partition $v$ is in. In other words, if $v \in P_{i,j} \in \mathcal{P}_i$ then one of the colors assigned to $v$ will be $c_{i,j}$.

This gives a proper $t$-tone coloring. Indeed, since each of the $t$ partitions respects all the others, any two vertices $u,v$ share at most one color, and if they do share one color then they are not adjacent because each partition consists of independent sets. 

To show that the $\mathcal{P}_1 \ldots \mathcal{P}_t$ exist w.h.p., we use induction on $t$. Suppose we are given $\mathcal{P}_1 \ldots \mathcal{P}_{t-1}$. We will construct $\mathcal{P}_t$ iteratively using Lemma \ref{denselem2} and the following fact. 

\begin{fact}\label{manyverts} W.h.p. for every set $U \subset [n]$ of size $|U| \ge \frac{n}{\ln^2 n}$, $G[U]$ has an independent set of size at least $s_0 = (1-o(1)) \a(G)$ that respects $\mathcal{P}_1 \ldots \mathcal{P}_{t-1}$.
 \end{fact}

Assuming this, we construct $\mathcal{P}_t$ by iteratively applying Fact \ref{manyverts}, removing independent sets until there are fewer than $\frac{n}{\ln^2 n}$ vertices remaining, at which point we apply Lemma \ref{denselem2} to greedily finish constructing the partition $\mathcal{P}_t$, as was done in Section 3.
\end{proof}

\begin{proof}[Proof sketch for Fact \ref{manyverts}]

This is analogous to Lemma \ref{denselem1}.  Janson's inequality gives an exponential bound on the probability that $G_{n,p}$ has no independent sets of size $k$ respecting some fixed partitions $\mathcal{P}_1 \ldots \mathcal{P}_{t-1}$. We let $B$ be the number of tuples $(\mathcal{P}_1, \ldots, \mathcal{P}_{t-1}, U)$ of partitions $\mathcal{P}_i$ and sets $U$ for which Fact \ref{manyverts} fails. We can then bound $\Mean{B}$ using a union bound, linearity of expectation, and Janson's inequality. 
 
\end{proof}

\subsection{Sparse Case}

Our precise result for $\t_2(G_{n,c/n})$ relied on the precise result that \[\t_2(T) =\ceil{\frac{\sqrt{8\D + 1} + 5}{2}}\]for any tree $T$. The $t$-tone chromatic number of trees is only known up to a constant factor. We will use the following result of Cranston, Kim and Kinnersly:
\begin{theorem}[Theorem 2 in \cite{CKK}]\label{treegenthm} For any integer $t\ge3$, there exist constants $c_1,c_2$ such that for any tree $T$, \[c_1\sqrt{\D(T)} \le \t_t(T) \le c_2\sqrt{\D(T)}.\] 
\end{theorem}

This theorem allows us to prove our result for sparse graphs:
\begin{maintheorem}\label{sparsemaingen}
 Let $G=G_{n,p}$ where $p=c/n$ with $c$ constant and let $t\ge 3$ be an integer. If we let $\D=\D(G)$ represent the maximum degree, then there exist constants $c_1,c_2$ such that w.h.p.,
\[c_1\sqrt{\D} \le \t_t(G_{n,p})\le c_2\sqrt{\D}\]
\end{maintheorem}

\begin{proof}[Proof sketch for Theorem \ref{sparsemaingen}]
The proof of this theorem is a generalization of the proof of Theorem \ref{sparsemainthm}. The main step in that proof was to prove that $G[V_2]$ is a forest. To prove this result, we will prove that $H_t:=G[V_{2t-2}]$ is a forest. One may check that the proof of Lemma \ref{HisForest} works in the same way for $H_t$. For example, in property $\ref{noHighlyConnectedY}$, we must replace $G^{5}$ with $G^{4t-3}$. For the size of the maximum component, we will get $(\ln n)^{(4t+9)/8}$. Then in the calculation for $\Pr{\neg \tbf{P}\mid \tbf{d}}$, the exponent of $57/8$ will be replaced by a higher constant depending on $t$. However $\exp\set{-\ln^{1/4}n}$ goes to zero fast enough to handle any polylog factor.

Since $H_t$ is a forest, we may $t$-tone color it with $\k:=\t_t(H_t) =\Theta\of{\sqrt{\D}}$ many colors by Theorem \ref{treegenthm}. We then remove the labels except for those on $V_{t-1}$. This proper $t$-tone coloring on $G[V_{t-1}]$ may be extended to a proper coloring of $G$ in the same way. We took care to ensure that any two vertices of $V_{t-1}$ which are at distance at most $t$ in $G$ receive appropriate labels. We may now show that the remaining vertices may be greedily colored using no new colors. We do this in the same way, by ensuring that the maximum number of forbidden labels at any uncolored vertex is much smaller than the number of total labels. In this case, we see that the number of forbidden labels is bounded above by
\[\sum_{i=1}^{t}b_0^i\binom{t}{i}\binom{\k}{t-i} = O\of{b_0\cdot\k^{t-1}} \ll \binom{\k}{t}.\]
\end{proof}

\end{document}